\documentclass[12pt]{article}
\usepackage{amsmath,amsfonts,amssymb,amsthm} 
\newtheorem{theorem}{Theorem}

\newtheorem*{prop2}{Lagrange lemma}
\newtheorem*{prop1}{Murnaghan rule}

\newtheorem{lem}{Lemma}

\theoremstyle{remark}
\newtheorem{remark}{Remark}
\numberwithin{equation}{section}
\newcommand\la{\lambda}

\newcommand\nbi[3]{{\binom{#1}{#2}}_{#3}}

\newcommand\npbi[3]{{\genfrac{\langle}{\rangle}{0pt}{}{#1}{#2}}_{#3}}

\title{An explicit formula\\
for the characters of the symmetric group}
\author{Michel Lassalle\\
\small Centre National de la Recherche Scientifique\\[-0.8ex]
\small Institut Gaspard-Monge, Universit\'e de Marne-la-Vall\'ee\\[-0.8ex]
\small 77454 Marne-la-Vall\'ee Cedex, France\\[-0.8ex]
\small \texttt{lassalle @ univ-mlv.fr}\\[-0.8ex]
\small \texttt{http://igm.univ-mlv.fr/{\textasciitilde}lassalle}}
\date{}

\begin{document}
\maketitle
\begin{abstract}
We give an explicit expression of the normalized characters of the symmetric group in terms of the ``contents" of the partition labelling the representation.
\end{abstract}

\section{Introduction}

The characters of the irreducible representations of the symmetric group play an important role in many areas of mathematics. However, since the early work of Frobenius~\cite{Fr} in 1900, no explicit formula was found for them. The characters of the symmetric group were computed through various recursive algorithms, but explicit formulas were only known for about ten particular cases~\cite{Fr,In}. The purpose of this paper is to give such an explicit expression in the general case. 

The irreducible representations of the symmetric group $S_n$ of $n$ letters are labelled by partitions $\la$ of $n$ (i.e. weakly decreasing sequences of positive integers summing to $n$). Their characters $\chi^\la$ are evaluated at a conjugacy class of $S_n$, labelled by a partition $\mu$ giving the cycle-type of the class. 
Let $\chi^\la_\mu$ be the value of the character $\chi^\la$ at a permutation of cycle-type $\mu$. We shall give an explicit formula for  
the normalized character $\hat{\chi}^\la_\mu=\chi^\la_\mu/\textrm{dim}\,\la$. This result was announced in~\cite{La5}.

It should be first emphasized that our formula gives the dependence of $\hat{\chi}^\la_\mu$ with respect to $\la$ in terms of the ``contents'' of this partition. More precisely the normalized character $\hat{\chi}^\la_\mu$ is expressed as some (unique) symmetric function evaluated on the contents of $\la$. 

This description of characters by content evaluation was proved in~\cite{KO} and~\cite{Co}. Previously the importance of contents had been apparent from the works of Jucys~\cite{Ju} and Murphy~\cite{My}. The fact had been noticed by Suzuki~\cite{Su}, Lascoux and Thibon~\cite{LT} and Garsia~\cite{Ga}.  
\newpage
Later tables and conjectures were independently given by Katriel~\cite{Ka} and by the author~\cite[Sections 8-11]{La1}. Actually the conjectures of~\cite{La1} were formulated in the framework of Jack polynomials. But, as emphasized in Section $11$ of~\cite{La1}, once specialized to $\alpha=1$, they correspond to the characters of the symmetric group.

However the symmetric function expressing $\hat{\chi}^\la_\mu$ remained quite obscure, even in the very elementary situation of a partition $\mu$ having only one non-unary part. The purpose of this paper is to give an explicit expression.

It is a second remarkable fact that this symmetric function can only be written by using a new family of positive integers, which we have introduced in~\cite{La2}. The connection of these integers with the symmetric group is still mysterious and certainly needs more investigation.

We emphasize that our method provides a very efficient algorithm, implemented on computer. Tables giving  $\hat{\chi}^\la_\mu$ for $|\mu|-l(\mu)\le 12$ will be available on a web page~\cite{W}.

\section{Notations}

We briefly recall some basic notions about the characters of the symmetric group, referring the reader to~\cite{Go} and to~\cite[Section 1.7]{Ma} for an elementary introduction.

\subsection{Characters}

A partition $\la= (\la_1,...,\la_r)$
is a finite weakly decreasing
sequence of nonnegative integers, called parts. The number
$l(\la)$ of positive parts is called the length of
$\la$, and $|\la| = \sum_{i = 1}^{r} \la_i$
the weight of $\la$. For any integer $i\geq1$,
$m_i(\la) = \textrm{card} \{j: \la_j  = i\}$
is the multiplicity of the part $i$ in $\la$.  Clearly
$l(\la)=\sum_{i\ge1} m_i(\la)$ and
$|\la|=\sum_{i\ge1} im_i(\la)$. We shall also write
$\la= (1^{m_1},2^{m_2},3^{m_3},\ldots)$. We set
\[z_\la  = \prod_{i \ge  1} i^{m_i(\lambda)} m_i(\lambda) ! .\]
We identify $\la$ with its Ferrers diagram  
$\{ (i,j) : 1 \le i \, \le l(\la), 1 \le j \le {\la}_{i} \}$.

Let $n$ be a fixed positive integer and $S_n$ the group of permutations of $n$ letters. Each permutation $\sigma \in S_n$ factorizes uniquely as a product of disjoint cycles, whose respective lengths are ordered such as to form a partition $\mu=(\mu_1,\ldots,\mu_r)$ with weight $n$. This partition is called the cycle-type of $\sigma$ and determines each permutation up to conjugacy in $S_n$. Conjugacy classes are thus labelled by partitions $\mu$ with $|\mu|=n$. 

The irreducible representations of $S_n$ and their corresponding characters are also labelled by partitions $\la$ with weight $|\la|=n$. We write ${\chi}^\la_\mu$ for the value of the character $\chi^\la(\sigma)$ at any permutation $\sigma$ of cycle-type $\mu$.

The dimension $\textrm{dim}\,\la=\chi^\la_{1^n}$ of the representation $\la$ is well known, see~\cite[Example 1.7.6, p.~116]{Ma} or~\cite[p.~54]{Go}. We have
\[\textrm{dim}\,\la=  \frac{n!}{\prod_{i=1}^{l(\la)} (\la_i+l(\la)-i)!} \prod_{1\le i<j\le l(\la)}(\la_i-\la_j+j-i).\]
We write $\hat{\chi}^\la_\mu=\chi^\la_\mu/\textrm{dim}\,\la$ for the corresponding normalized character.

\subsection{Symmetric functions}

Let $A=\{a_1,a_2,a_3,\ldots\}$ a (possibly infinite) set of
independent indeterminates, called an alphabet. The generating functions
\begin{equation*}
E_z(A)=\prod_{a\in A} (1 +za) =\sum_{k\geq0} z^k\, e_k(A),\qquad
H_z(A)=\prod_{a\in A}  (1-za)^{-1} = \sum_{k\geq0} z^k\, h_k(A)
\end{equation*}
define symmetric functions known as respectively elementary and complete. The power sum symmetric functions
are defined by $p_{k}(A)=\sum_{i \ge 1} a_i^k$. For any partition $\mu$, we define functions $e_\mu$, 
$h_\mu$ or $p_\mu$ by
\[f_{\mu}= \prod_{i=1}^{l(\mu)}f_{\mu_{i}}=\prod_{k\geq1}f_k^{m_{k}(\mu)},\]
where $f_{i}$ stands for $e_i$, $h_i$  or $p_i$. 

When $A$ is infinite, each of the three sets of functions $e_i$, 
$h_i$ or $p_i$ forms an algebraic basis of
$\mathcal{S}$, the symmetric algebra with coefficients in $\mathsf{R}$. Each of the sets of functions $e_\mu$, $h_\mu$, $p_\mu$ is a linear basis of this algebra. 

Another linear basis is formed by the Schur functions $s_{\la}$, which are defined by the Jacobi-Trudi formula
\begin{equation*}
s_\la=\det_{1\le i,j \le l} \,[h_{\la_i-i+j}],
\end{equation*}
with $h_i=0$ for $i<0$. This definition is usually written for a partition $\la$ with length $l$. However it remains valid when $\la$ is replaced by any sequence of integers $\la \in \mathsf{Z}^l$, not necessarily in descending order. Then using the obvious rule
\[s_{\ldots,a,b,\ldots}=-s_{\ldots,b-1,a+1,\ldots}\]
for $b>a$, it is easily seen that $s_{\la}$ is either $0$, either equal to $\pm s_{\mu}$, with $\mu$ a partition.

\subsection{Shifted symmetric functions}

Although the theory of symmetric functions goes back to the early 19th century, shifted symmetric functions are quite recent. They were introduced and studied in~\cite{KO,OO}.

Being given a finite alphabet  $A=\{a_1,a_2,\ldots,a_r\}$, a polynomial in $A$ is ``shifted symmetric'' if it is symmetric in the shifted variables $a_i-i$. When $A=\{a_1,a_2,a_3,\ldots\}$ is infinite, in analogy with symmetric functions, 
a ``shifted symmetric function'' $f$ is a family $\{f_i, i\ge 1\}$ such that  $f_i$ is shifted symmetric in $(a_1,a_2,\ldots,a_i)$, together with the stability property $f_{i+1}(a_1,a_2,\ldots,a_i,0)=f_i(a_1,a_2,\ldots,a_i)$.  

This defines $\mathcal{S}^{\ast}$, the shifted symmetric algebra with coefficients in $\mathsf{R}$, which is algebraically generated by the ``shifted power sums''
\[p_k^{*}(A)=\sum_{i\ge 
1}\Big((a_i-i+1)_k-(-i+1)_k\Big).\]
Here for an indeterminate $z$ and any positive integer $p$, the \textit{lowering} factorial 
\[(z)_p = z(z-1) \ldots (z-p+1)=\sum_{i=1}^p s(p,i) \,z^i,\]
is the generating function of the Stirling numbers of the first kind $s(p,i)$. Conversely 
\[z^p=\sum_{i=1}^p S(p,i)(z)_i\]
is the generating function of the Stirling numbers of the second kind $S(p,i)$.

An element $f\in \mathcal{S}^{\ast}$ may be evaluated at any sequence 
$(a_1,a_2,\ldots)$ with finitely many non zero terms, hence at any partition $\la$. Moreover by analyticity, $f$ is entirely determined by its restriction $f(\la)$ to partitions. This identification is usually performed and $\mathcal{S}^{\ast}$ is considered as a function 
algebra on the set of partitions.

\subsection{Contents}

Given a partition $\la$, the ``content'' of any node $(i,j) \in \la$ is defined as $j-i$. Denote $A_\la=\left\{j-i,\, (i,j) \in \la \right\}$ the finite alphabet of the contents of $\la$. The symmetric algebra $\mathcal{S}[A_\la]$ is generated by the power sums
\[p_k(A_\la) = \sum_{(i,j) \in \la} (j-i)^k=
\sum_{i=1}^{l(\la)} \sum_{j=1}^{\la_i} (j-i)^k.\]

It is well known~\cite{KO,OO} that the quantities $p_k(A_\la)$ are shifted symmetric polynomials of $\la$. Indeed for any integer $k\ge 1$, applying the identity $r(z)_{r-1}=(z+1)_{r}-(z)_{r}$, we have
\begin{equation*}
\begin{split}
p_k(A_\la)&= \sum_{r=1}^k \sum_{(i,j)\in \la} S(k,r)\,
(j-i)_{r}\\
&= \sum_{r=1}^k \frac{S(k,r)}{r+1} \sum_{i= 1}^{l(\la)} 
\Big((\la_i-i+1)_{r+1}-(-i+1)_{r+1}\Big),\\
&=\sum_{r=1}^k  \frac{S(k,r)}{r+1}\, p_{r+1}^{*}(\la).
\end{split}
\end{equation*}
Hence the statement. For instance we have
\[p_1(A_\la)=\frac{1}{2}p_2^{*}(\la),\qquad p_2(A_\la)=
\frac{1}{3}p_3^{*}(\la)+\frac{1}{2}p_2^{*}(\la),\qquad
p_3(A_\la)= \frac{1}{4}p_4^{*}(\la)+p_3^{*}(\la)+\frac{1}{2}p_2^{*}(\la).\]

As a straightforward consequence, the shifted symmetric algebra $\mathcal{S}^{\ast}$ is algebraically generated by the functions $p_k(A_\la), k\ge 1$ together with $p_1^{*}(\la)=|\la|$. The latter corresponds to the cardinal of the alphabet $A_\la$. 

Any shifted symmetric function may be written $f(A_\la)$, with $f \in \mathsf{R}[\textrm{card},p_1,p_2,p_3,\ldots]$. Moreover this expression is unique. From now on we shall abbreviate 
\[p_k(\la):=p_k(A_\la),\qquad p_\mu(\la):=p_\mu(A_\la).\]
This notation will not bring any confusion with the power sum $p_k(\la_1,\la_2,\ldots)=\sum_{i\ge 1} \la_i^k$, which is never used in the sequel. 

\subsection{Murnaghan rule}

The transition matrices between Schur functions and power sums are given by the Frobenius formulas
\begin{equation*}
s_{\la}=\sum_{\mu} z_\mu^{-1}\chi^\la_\mu \, p_\mu, \qquad
p_{\mu}=\sum_{\la}\chi^\la_\mu \, s_\la.
\end{equation*}

Both formulas remain valid when the partition $\la$, with length $l$, is replaced by any sequence of integers $\la \in \mathsf{Z}^l$. This allows to define a generalized (or virtual) character $\chi^\la$ for any sequence of integers $\la \in \mathsf{Z}^l$.

For such $\la$, given some positive integer $r$, we have~\cite[Example 1.3.11, p.~48]{Ma}
\begin{equation*}
p_r \, s_{\la}=\sum_{i=1}^{l+1} s_{\la + r\epsilon_i},
\end{equation*}
with $\epsilon_i$ the sequence having $1$ in the $i$-th place and $0$ elsewhere.

The following recurrence property is a straightforward consequence. Given a partition $\la$ with length $l(\la)$, and two integers $k,p$ with $k\le l(\la)$, let $\la -p\epsilon_k$ be the multi-integer $(\la_1,\ldots,\la_{k-1},\la_k-p,\la_{k+1},\ldots,\la_{l(\la)})\in \mathsf{Z}^{l(\la)}$.

\begin{prop1}
Let $\la,\mu$ be two partitions with weight $n$. Let $p$ be some part of $\mu$, and ${\mu \setminus p}$ the partition obtained by substracting $p$ from $\mu$. Then we have
\begin{equation*}
\chi^\la_\mu=\sum_{k=1}^{l(\la)} \chi^{\la -p\epsilon_k}_{\mu \setminus p}.
\end{equation*}
\end{prop1}

In general the quantities $\chi^{\la -p\epsilon_k}$ appearing on the right-hand side are not characters, but virtual characters. In the usual formulation of the rule~\cite[Example 1.7.5, p.~117]{Ma}, these virtual characters are expressed in terms of characters. This is Nakayama's version~\cite{Na} of Murnaghan's formula~\cite{Mu}. Our method has the advantage of making this operation totally unnecessary.
\newpage
We shall need a formulation of the previous rule in terms of normalized characters. Given a partition $\la$ with weight $n$ and length $l(\la)$, and two integers $k,p$ with $k\le l(\la)$, we write
\begin{equation*}
\begin{split}
d_\la(k,p)&=\frac{n!}{(n-p)!}\,\frac{\textrm{dim}\,(\la -p\epsilon_k)}{\textrm{dim}\,\la}\\
&=\frac{(\la_k+l(\la)-k)!}{(\la_k+l(\la)-k-p)!} 
\prod_{\begin{subarray}{c}i=1\\i\neq k\end{subarray}}^{l(\la)}
\frac{\la_k-\la_i+i-k-p}{\la_k-\la_i+i-k}.
\end{split}
\end{equation*}
The Murnaghan rule then writes
\begin{equation*}
(n)_p\,\hat{\chi}^\la_\mu=\sum_{k=1}^{l(\la)} d_\la(k,p) \, \hat{\chi}^{\la -p\epsilon_k}_{\mu \setminus p}.
\end{equation*}
This recurrence relation is our first ingredient for the computation of $\hat{\chi}^\la_\mu$.

\section{Lagrange interpolation}

Our second ingredient is Lagrange interpolation, written under the following form~\cite{Ls}.
For any two alphabets $A$ and $B$, their difference $A-B$ 
(which is not their difference as sets) is defined by 
\[H_z(A-B) = H_z(A)\,{H_z(B)}^{-1}=\frac {\prod_{b\in B} (1-zb)}{\prod_{a\in A} (1-za)}.\]
Alain Lascoux~\cite{Ls} mentions that when $B$ is empty, 
the following result was already known to Euler.
\begin{prop2}
Let $A$ and $B$ be two finite alphabets with respective 
cardinals $m$ and $n$. For any integer $r\ge0$ we have
\[\sum_{a \in A} a^r \, \frac{\displaystyle\prod_{b \in B}(a-b)}
{\displaystyle\prod_{c \in A, \, c \neq a} (a-c)} =
h_{n-m+r+1}(A-B).\]
\end{prop2}

Given a partition $\la$, let $r$ be the number of nodes in the main diagonal of its Ferrers diagram, and $\alpha_i$ (resp. $\beta_i$) be the number of nodes in the $i$-th row (resp. column) on the right of (resp. below) the node $(i,i)$. The couples $(\alpha_i,\beta_i), i=1\ldots r$ are known as Frobenius coordinates. 

The ``Frobenius function'' is defined by
\begin{equation*}
F(z;\la)=\prod_{i=1}^r \frac{z-\alpha_i}{z+\beta_i+1}
=\prod_{i\ge 1} \frac{z-\la_i+i}{z+i}.
\end{equation*}
The equality goes back to Frobenius~\cite{Fr}, see~\cite[Example 1.1.15, p.~17]{Ma} or~\cite{IO}. 

For any indeterminate $z$ and positive integer $p$, we consider the function
\begin{equation*}
(z)_p \frac{F(z-p;\la)}{F(z;\la)}= (z)_p \prod_{i=1}^{l(\la)}
\frac{z-\la_i+i-p}{z+i-p}\, \frac{z+i}{z-\la_i+i},
\end{equation*}
and its expansion in descending powers of $z$, i.e. its Taylor series at infinity,
\begin{equation*}
(z)_p \frac{F(z-p;\la)}{F(z;\la)}=\sum_{r\ge -p} C_r(\la;p) z^{-r}.
\end{equation*}

\begin{theorem}
For any $r\ge 0$ we have
\begin{equation*}
C_{r+1}(\la;p)=-p\sum_{k=1}^{l(\la)} d_\la(k,p) (\la_k-k)^r.
\end{equation*}
\end{theorem}
\begin{remark} 
For $r=0$, this result is due to Frobenius~\cite{Fr}, see~\cite[Example 1.7.7, p.~118]{Ma}. The cases $r=1,2$ were investigated by Ingram~\cite{In}.
\end{remark}
\begin{proof}
We apply the Lagrange lemma for the two following alphabets
\begin{equation*}
\begin{split}
A&=\{a_i =\la_i-i,\quad i=1,\ldots,l(\la) \},\\
B&=\{b_i =\la_i-i+p,\quad i=1,\ldots,l(\la)+p\}\end{split}
\end{equation*}
We have
\begin{equation*}
\begin{split}
H_z(A-B)&=\prod_{i=1}^{l(\la)} \frac{1-z(\la_i -i+p)}{1-z(\la_i-i)} \, 
\prod_{i=1}^{p} \big(1+z(l(\la)-p+i)\big)\\
&=\frac{F(1/z-p;\la)}{F(1/z;\la)}\, \prod_{i=1}^{l(\la)} \frac{1+z(i-p)}{1+zi} \, \prod_{i=1}^{p} \big(1+z(l(\la)-p+i)\big)\\
&=\frac{F(1/z-p;\la)}{F(1/z;\la)}\,\prod_{i=1}^{p} \big(1+z(i-p)\big)\\
&=z^p\, (1/z)_p\, \frac{F(1/z-p;\la)}{F(1/z;\la)}.
\end{split}
\end{equation*}
Therefore
\begin{equation*}
(z)_p \frac{F(z-p;\la)}{F(z;\la)}=z^p H_{1/z}(A-B)
=\sum_{k\geq 0} z^{p-k}\, h_k(A-B)
=\sum_{r\geq -p} z^{-r}\, h_{r+p}(A-B).
\end{equation*}
On the other hand, it is obvious that
\[-p\,d_\la(k,p)= \frac {\displaystyle \prod_{b \in B}(a_k-b)}
{\displaystyle \prod_{c \in A,\, c \neq a_k}(a_k-c)},\]
and for any $r\ge 0$ the Lagrange lemma writes
\[-p\sum_{k=1}^{l(\la)} d_\la(k,p) (\la_k-k)^r= h_{r+p+1}(A-B).\]
Hence the statement.
\end{proof}

\section{Explicit series expansion}

Our third ingredient is a formula giving explicitly the Taylor series at infinity of
\[(z)_p \frac{F(z-p;\la)}{F(z;\la)}.\]
Some preliminary results are necessary.

\subsection{Positive integers}

Let $n,p,k$ be three integers with $0 \le p \le n$ and $k \ge 1$. Define
\begin{equation*}
{\binom{n}{p}}_{k} 
= \frac{n}{k} \ \sum_{r \ge 0} 
\binom{p}{r}\binom{n-p}{r} 
\binom{n-r-1}{k-r-1}.
\end{equation*}
We have obviously
\[{\binom{n}{p}}_{k} = 0 \quad \textrm{for} 
\quad k>n, \quad\quad 
\nbi{n}{p}{1}=n,  
\quad\quad  
{\binom{n}{p}}_{k} ={\binom{n}{n-p}}_{k} .\]
These numbers generalize 
the classical binomial coefficients, since we have
\[\nbi{n}{0}{k} = \binom{n}{k}, \quad\quad \nbi{n}{1}{k}=k \binom{n}{k},
\quad\quad \nbi{n}{p}{n} = \binom{n}{p},\]
the last property being a direct consequence of the classical 
Chu-Vandermonde formula.

The numbers $\nbi{n}{p}{k}$ were studied in~\cite{La2}. It was proved that they  are \textit{positive integers}, and their generating function
\[G_n(y,z) =\sum_{p=0}^{n} \sum_{k=1}^{n} \nbi{n}{p}{k} y^p z^k\]
was shown to be~\cite{Z}
\begin{multline*}
G_n(y,z)= 2^{-n} {\left( (1+y)(1+z)+\sqrt{(1+y)^2(1+z)^2-4y(1+z)} 
\right)}^n\\
+ 2^{-n} {\left( (1+y)(1+z)-\sqrt{(1+y)^2(1+z)^2-4y(1+z)} \right)}^n -1-y^n.
\end{multline*}

\subsection{Extension to partitions}

For any integers $0 \le p \le |\la|$ and $k \ge 1$, we define
\[\npbi{\la}{p}{k}= \sum_{(p_i)} \sum_{(k_i)} \prod_{i=1}^{l(\la)} 
\nbi{\la_i}{p_i}{k_i},\]
the sum being taken over all decompositions $p=\sum_{i=1}^{l(\la)} p_i$, 
$k=\sum_{i=1}^{l(\la)} k_i$ with $0 \le p_i \le \la_i$ and 
$k_i \neq 0$ for any $i$. 
Observe that there is no such restriction for $p_i$. 

This definition yields easily
\[\npbi{\la}{p}{k}=0 \quad \textrm{except if} \quad l(\la) \le k \le |\la|.\]
Indeed it is obvious that $\npbi{\la}{p}{k}=0$ for $k<l(\la)$, and since 
$\nbi{n}{p}{k} = 0$ for $k>n$, we have also $\npbi{\la}{p}{k}=0$ for
$k>|\la|$.

For instance $\npbi{\la}{p}{1}=0$ except if $\la$ is a row partition $(n)$, 
in which case $\npbi{(n)}{p}{k}=\nbi{n}{p}{k}$.
We have easily
\[\npbi{\la}{1}{k}=k\npbi{\la}{0}{k},\quad\quad \npbi{\la}{p}{k}=\npbi{\la}{|\la|-p}{k},\quad\quad \npbi{\la}{p}{|\la|}=\binom{|\la|}{p}.\]
As a direct consequence of their definition, 
the generating function for the positive integers $\npbi{\la}{p}{k}$ is the following
\begin{equation*}
\sum_{p=0}^{|\la|} \sum_{k=l(\la)}^{|\la|} 
\npbi{\la}{p}{k} y^p z^k 
=\prod_{i=1}^{l(\la)} G_{\la_i}(y,z)
=\prod_{i \ge 1} {\Big(G_i(y,z) \Big)}^{m_i(\la)}.
\end{equation*}

\subsection{New symmetric functions}

For any integers 
$n \ge 1$, $k \ge 1$ and $0 \le p \le n$, we define the symmetric function
\[F_{npk} =\sum_{|\mu| = n} \frac{\npbi{\mu}{p}{k}}{z_{\mu}}\, p_\mu.\]
Since $\npbi{\mu}{p}{k}=0$ for $k<l(\mu)$, this sum is restricted to partitions 
with $l(\mu) \le k$. Similarly since $\npbi{\mu}{p}{k}=0$ for $k>|\mu|$,
one has $F_{npk}=0$ for $k>n$.

For $k=0$ the previous definition is extended by the convention
$F_{np0}=0$ with the only exception $F_{000}=1$.
For $k=1$ and any $p \le n$ we have $F_{np1}=p_n$. For $k=n$ we obtain
\[F_{npn}= \binom{n}{p} \sum_{|\mu| = n}
\frac{p_\mu}{z_{\mu}}=\binom{n}{p} F_{n0n}=\binom{n}{p} h_n,\]
where we have used~\cite[p.~25]{Ma}. We have also $F_{n1k}= 
kF_{n0k}$ and $F_{npk}=F_{n,n-p,k}$.

From now on we abbreviate $F_{npk}(\la):=F_{npk}(A_\la)$, the value of the symmetric function $F_{npk}$ on the alphabet $A_\la$. As a consequence of Section 2.4, it is a shifted symmetric function of $\la$.

\subsection{Taylor expansion}

For any partition $\la$ the ``content polynomial'' of $\la$ is defined by
\[C_{\la}(z)=\prod_{(i,j) \in \la} \left(z+j-i\right).\]
Since
\[\frac {C_{\la}(z)}{C_{\la}(z-1)}=\prod_{i=1}^{l(\la)}
\prod_{j=1}^{\la_i} \frac{z+j-i}{z+j-i-1}
=\prod_{i=1}^{l(\la)} \frac{z+\la_i-i}{z-i},\]
we have
\[
\frac{F(z-p;\la)}{F(z;\la)}= \frac {C_{\la}(-z+p)}{C_{\la}(-z+p-1)}
\frac {C_{\la}(-z-1)}{C_{\la}(-z)}.\]

Given two indeterminates $x,y$, the expansion of 
\[\frac{C_{\la}(x-y+1)}{C_{\la}(x-y)}\, 
\frac{C_{\la}(-y)}{C_{\la}(-y+1)} =-x
\sum_{r\ge 0} c^\la_r(x)\, y^{-r}\]
in descending powers of $y$ was explicitly obtained in~\cite[Corollary~5.2, p.~3464]{La3}, in the more general context of ``$\alpha$-contents''. The coefficients $c^\la_r(x)$ are given by
\begin{multline*}
c^\la_r(x)=\\
\sum_{\begin{subarray}{c}i,j,m\ge 0 \\ 2m+i+j \le r\end{subarray}} 
(-x)^{m-1} (x+1)^i\binom{m+i+j-1}{i}
\Bigg(\sum_{k=0}^{\mathrm{min}(m,r-2m-i)}
\binom{|\la|+m-1}{m-k} F_{r-2m-i,j,k}(\la) \Bigg).
\end{multline*}

We have $c^\la_0(x)=-1/x$ since $F_{000}(\la)=1$, and $c^\la_1(x)=0$ since $F_{1p0}(\la)=0$. With $\hat x=x+1$, first values are given by
\begin{equation*}
\begin{split}
c^\la_2(x)&=|\la|, \quad\quad c^\la_3(x)=2p_1(\la)+\hat x|\la|,\\
c^\la_4(x)&=3p_2(\la)+3\hat xp_1(\la)-x\binom{|\la|+1}{2}+\hat x^2|\la|,\\
c^\la_5(x)&=4p_3(\la)+6\hat xp_2(\la)-x(|\la|+1)\Big(2p_1(\la)+\hat x|\la|\Big)+4\hat x^2p_1(\la)+\hat x^3|\la|.
\end{split}
\end{equation*}

From Section 2.4 it is obvious that any $c^\la_r(x)$ is a shifted symmetric function of $\la$. These auxiliary functions will be our main tool in the sequel.

\begin{theorem}
We have the following Taylor series at infinity
\[ (z)_p \frac{F(z-p;\la)}{F(z;\la)}= -p\sum_{r,s,t\ge 0} (-1)^s \binom{r+s-1}{s} s(p,t)\, c^\la_r(p)\, z^{t-r-s}.\]
\end{theorem}
\begin{proof}
We have
\begin{equation*}
\frac{F(z-p;\la)}{F(z;\la)}\,
=-p\sum_{r\ge 0} c^\la_r(p) (1+z)^{-r}\,
=-p\sum_{r,s\ge 0} (-1)^s \binom{r+s-1}{s}c^\la_r(p)\, z^{-r-s}.
\end{equation*}
\end{proof}

The following result asserts that some \textit{rational} function of $\la$ (defined at the left-hand side) is actually a shifted symmetric \textit{polynomial}. 

\begin{theorem}
For any $r\ge 0$ we have
\begin{equation*}
\sum_{k=1}^{l(\la)} d_\la(k,p) (\la_k-k)^r=(-1)^r
\sum_{i,j\ge 0} (-1)^j \binom{r}{j} s(p+1,i-j)\,c^\la_i(p). 
\end{equation*}
\end{theorem}
\begin{proof}
By Theorem 1 the left-hand side is the coefficient of $z^{-r-1}$ in the Taylor series expansion of $(-1/p)(z)_p F(z-p;\la)/F(z;\la)$. By Theorem 2 this coefficient evaluates as
\[\sum_{i,j\ge 0} (-1)^j \binom{i+j-1}{j} s(p,i+j-r-1) c^\la_i(p)\]
But we have
\[\sum_{j\ge 0} (-1)^{j} \binom{i+j-1}{i-1} s(p,i+j-r-1) =
(-1)^r \sum_{k\ge 0} (-1)^{k} \binom{r}{k} s(p+1,i-k).\]
This is a direct consequence of the Chu-Vandermonde identity
\[\binom{i+j-1}{i-1}=\sum_{k\ge 0}\binom{r}{k}\binom{i+j-r-1}{i-k-1},\]
together with the easy $r=0$ case
\[\sum_{m\ge 0} \binom{m-1}{k-1} |s(p,m-1)| =
|s(p+1,k)|.\]
\end{proof}

\section{One non-unary cycle}
 
Let $|\la|=n$, $\rho$ a partition with $|\rho|\le n$ and  $\mu=(\rho,1^{n-|\rho|})$. A fundamental result~\cite{VK,KO,IO} asserts that the normalized character $(n)_{|\rho|} \hat{\chi}^\la_\mu$ is a shifted symmetric function of $\la$. Our purpose is to express this character in terms of the auxiliary shifted symmetric functions $c^\la_r(p)$ previously introduced.

The proof is done by recurrence over the number of parts of $\mu$ distinct from $1$, i.e. the number of non-unary cycles of permutations having cycle-type $\mu$.

We begin with permutations which are $p$-cycles, i.e. having only one cycle of length $p$ and all other cycles unary. 

\begin{theorem} 
For $\mu=(p,1^{n-p})$ we have
\[(n)_p\, \hat{\chi}^\la_\mu=\sum_{i\ge 2} s(p+1,i)\, c^\la_i(p).\]
\end{theorem}
\begin{proof}
By the Murnaghan rule, we have
\begin{equation*}
(n)_p\, \hat{\chi}^\la_\mu=\sum_{k=1}^{l(\la)} d_\la(k,p).
\end{equation*}
We apply Theorem $3$ with $r=0$.
\end{proof}

For $p=1$ the sum is restricted to $i=2$ and we recover $\hat{\chi}^\la_{1^n}=1$. Frobenius~\cite{Fr} computed the three cases $2\le p\le 4$, Suzuki~\cite{Su} the cases $p=2, 3$ and Ingram~\cite{In} the case $p=5$. Their results may be recovered as follows :
\begin{equation*}
\begin{split}
(n)_2\, \hat{\chi}^\la_{2,1^{n-2}}&=
c_3^\la(2) - 3 c_2^\la(2)=2p_1(\la)\\
(n)_3\, \hat{\chi}^\la_{3,1^{n-3}}&=
c_4^\la(3) - 6 c_3^\la(3)+11 c_2^\la(3)=3p_2(\la)-3\binom{n}{2}\\
(n)_4\, \hat{\chi}^\la_{4,1^{n-4}}&=
c_5^\la(4) - 10 c_4^\la(4)+35 c_3^\la(4)-50c_2^\la(4)=4p_3(\la)-4(2n-3)p_1(\la)\\
(n)_5\, \hat{\chi}^\la_{5,1^{n-5}}&=c_6^\la(5)-15c_5^\la(5)+85 c_4^\la(5)-225c_3^\la(5)+274c_2^\la(5)\\&
=5p_4(\la)-5(3n-10)p_2(\la)-10p_1^2(\la)+25\binom{n}{3}-15\binom{n}{2}.
\end{split}
\end{equation*}

\begin{remark}
Since the left-hand side is a shifted symmetric \textit{polynomial} of $\la$, Theorem $4$ keeps true for virtual characters, i.e. when $\la$ is replaced by \textit{any sequence of integers}, not necessarily in descending order. This extension is obtained by analytic continuation.
\end{remark}

\section{Two non-unary cycles}

The second step of our recurrence corresponds to a permutation having only two non-unary cycles with lengths $p\ge q$, whose cycle-type is the partition $\mu=(p,q,1^{n-p-q})$. We shall need two auxiliary lemmas.
\begin{lem}
Given four positive integers $i,k,p,q$ with $k\le l(\la)$, we have
\[c^{\la -p\epsilon_k}_i(q)=c^{\la}_i(q)+
pq\sum_{\begin{subarray}{c}r,s,t\ge 0 \\ r+s+t+2=i\end{subarray}}
c^{\la}_t(q) (\la_k-k+q+1)^r (\la_k-k-p+1)^s.\]
\end{lem}
\begin{proof}
By analytic continuation, both sides being shifted symmetric polynomials of $\la$, we may assume that $\la -p\epsilon_k$ is a partition. By definition we have
\begin{equation*}
\begin{split}
\frac{F(z-q-1;\la -p\epsilon_k)}{F(z-1;\la -p\epsilon_k)}
\frac{F(z-1;\la)}{F(z-q-1;\la)}
&=\frac{z-\la_k+k+p-q-1}{z-\la_k+k-q-1}\, 
\frac{z-\la_k+k-1}{z-\la_k+k+p-1}\\
&=1+\frac{pq}{(z-\la_k+k-q-1)(z-\la_k+k+p-1)}\\
&=1+\frac{pq}{z^2}\sum_{r,s\ge 0}(\la_k-k+q+1)^r(\la_k-k-p+1)^s z^{-r-s}.
\end{split}
\end{equation*}
Expanding series and identifying coefficients, we can conclude.
\end{proof}

\begin{lem}
Given two indeterminates $x,y$  and two positive integers $a,b$, we have
\[\sum_{\begin{subarray}{c}r,s,u,v\ge 0\\ r+s=a\end{subarray}}
(-1)^{u+v}
\binom{r}{u}\binom{s}{v}\binom{u+v}{a-b}(x+1)^{r-u} (y+1)^{s-v}=
(-1)^{a+b}\binom{a+1}{b+1}\frac{x^{b+1}-y^{b+1}}{x-y}.\]
\end{lem}
\begin{proof}A direct consequence of the Chu-Vandermonde identity together with the two elementary properties
\begin{equation*}
\begin{split}
\sum_{u\ge 0}\binom{r}{u}\binom{u}{k}x^{r-u}&=\binom{r}{k}(1+x)^{r-k},\\
\sum_{\begin{subarray}{c}r+s=a\\ k+l=a-b\end{subarray}}
\binom{r}{k}\binom{s}{l}x^{r-k}y^{s-l}&=
\binom{a+1}{b+1}\frac{x^{b+1}-y^{b+1}}{x-y}.
\end{split}
\end{equation*}
\end{proof}
\begin{theorem} For $\mu=(p,q,1^{n-p-q})$ we have
\begin{multline*}
(n)_{p+q}\, \hat{\chi}^\la_\mu=\sum_{i,j \ge 2} c^\la_i(p)\, c^\la_j(q)\, s(p+1,i)\,s(q+1,j)\\
+pq\,\sum_{i,j \ge 0} c^\la_i(p)\, c^\la_j(q)
\Bigg(\sum_{a,b \ge 0} \binom{a+1}{b+1} 
\frac{p(-p)^{b}+q^{b+1}}{p+q} s(p+1,i-a+b)\,s(q+1,j+a+2)\Bigg).
\end{multline*}
\end{theorem}
\begin{proof}
Both sums are obviously finite. The second is restricted to $i+j\le p+q$ and $b\le a\le q-1$. The Murnaghan rule writes
\begin{equation*}
(n)_{p+q}\, \hat{\chi}^\la_{p,q,1^{n-p-q}}=\sum_{k=1}^{l(\la)} d_\la(k,p)\, (n-p)_{q}\,\hat{\chi}^{\la -p\epsilon_k}_{q,1^{n-p-q}}.
\end{equation*}
In general $\la -p\epsilon_k$ is not a partition. However taking Remark 2 into account, Theorem $4$ yields
\begin{equation*}
(n)_{p+q}\, \hat{\chi}^\la_{p,q,1^{n-p-q}}= \sum_{k=1}^{l(\la)} d_\la(k,p)\, \sum_{i\ge 2}  s(q+1,i)\, c^{\la -p\epsilon_k}_i(q).
\end{equation*}
By Lemma $1$ we have
\begin{multline*}
(n)_{p+q}\, \hat{\chi}^\la_{p,q,1^{n-p-q}}= \sum_{i\ge 2}  
s(q+1,i)\, c^{\la}_i(q)
\sum_{k=1}^{l(\la)} d_\la(k,p) \\ + pq
\sum_{k=1}^{l(\la)} d_\la(k,p)\, \sum_{i\ge 2}  s(q+1,i)\, \sum_{\begin{subarray}{c}r,s,t\ge 0 \\ r+s+t+2=i\end{subarray}}
c^{\la}_t(q) (\la_k-k+q+1)^r (\la_k-k-p+1)^s.
\end{multline*}
Applying Theorem $3$ with $r=0$, the first term on the right-hand side is clearly
\[\sum_{i,j \ge 2} c^\la_i(p)\, c^\la_j(q)\, s(p+1,i)\,s(q+1,j).\]
The second term can be written
\begin{equation*}
pq \sum_{r,s,t,u,v\ge 0} s(q+1,r+s+t+2)\, c^{\la}_t(q) \binom{r}{u}\binom{s}{v}(1+q)^{r-u} (1-p)^{s-v} \sum_{k=1}^{l(\la)} d_\la(k,p)\,(\la_k-k)^{u+v}.
\end{equation*}
Applying Theorem $3$ we obtain
\begin{multline*}
pq \sum_{j,r,s,u,v\ge 0} s(q+1,r+s+j+2)\, c^{\la}_j(q) \binom{r}{u}\binom{s}{v}(1+q)^{r-u} (1-p)^{s-v}\\ 
\times (-1)^{u+v}
\sum_{i,k\ge 0} (-1)^k \binom{u+v}{k} s(p+1,i-k)\,c^\la_i(p).
\end{multline*}
We conclude by using Lemma $2$.
\end{proof}

For $q=1$ the sums are restricted to $j=2$ (resp. $j=a=b=0$), and we recover Theorem $4$. Suzuki~\cite{Su} computed the case $p=q=2$, and Ingram~\cite{In} the four cases $2\le p\le 4, q=2$ and $p=q=3$ (with many misprints). 

We give the examples
\begin{equation*}
\begin{split}
(n)_6\, \hat{\chi}^\la_{3,3,1^{n-6}}=&
\left(c_4^\la(3)-6 c_3^\la(3)+ 11 c_2^\la(3)\right) \left(c_4^\la(3)-6 c_3^\la(3)+ 20 c_2^\la(3)\right)\\
& -9c_6^\la(3)+90c_5^\la(3)-375c_4^\la(3)+810c_3^\la(3)-876c_2^\la(3)
\end{split}
\end{equation*}
\begin{equation*}
\begin{split}
(n)_7\, \hat{\chi}^\la_{4,3,1^{n-7}}=&
\left(c_5^\la(4)-10 c_4^\la(4)+35 c_3^\la(4)-50 c_2^\la(4) \right) 
\left(c_4^\la(3)-6c_3^\la(3)+ 23c_2(3)\right)\\ &-12\left(c_7^\la(4) -15 c_6^\la(4) +95 c_5^\la(4)- 325 c_4^\la(4)  + 624 c_3^\la(4)- 620 c_2^\la(4)\right).
\end{split}
\end{equation*}

\begin{remark} 
The expression given by Theorem $5$ is not symmetrical with respect to $(p,q)$, though actually $\hat{\chi}^\la_\mu$ is. Thus the equivalence of formulas written for $(p,q)$ and for $(q,p)$ yields identities between the $c_j^\la$ 's.
The simplest case of such an identity is obtained for $\mu=(2,1^{n-2})$. Then writing Theorem $5$ for $(2,1)$ and for $(1,2)$, we have
\begin{equation*}
(n)_3\, \hat{\chi}^\la_{2,1^{n-2}}=(c_3^\la(2) - 3 c_2^\la(2))(c_2^\la(1)- 2) = - 2c_3^\la(1) + c_2^\la(1)(c_3^\la(2) -3c_2^\la(2)+ 4),
\end{equation*}
which gives 
\[3 c_2^\la(2)-2 c_2^\la(1) = c_3^\la(2)-c_3^\la(1)=n.\]
\end{remark}

\section{The general case}

Theorem $5$ may be written in a more compact form, by using the following notations. Let $\varepsilon \in \{0,2\}$. Define $\theta=1$ if $\varepsilon=0$ and $\theta=pq$ otherwise. Then Theorem $5$ reads
\begin{equation*}
(n)_{p+q}\, \hat{\chi}^\la_{p,q,1^{n-p-q}}=
\sum_{\varepsilon\in \{0,2\}}\, \sum_{i,j \ge 0} \,
A_{ij}^{(\varepsilon)}(p,q) \,c^\la_i(p)\, c^\la_j(q), 
\end{equation*}
with
\begin{equation*}
A_{ij}^{(\varepsilon)}(p,q)=
\sum_{a,b \ge 0}  \theta \, \binom{a+1}{b+1} 
\frac{p(-p)^{b}+q^{b+1}}{p+q}
s(p+1,i-a+b)\,s(q+1,j+a+\varepsilon),
\end{equation*}
and the convention that the sum on $a,b$ is restricted to $a=b=0$ when $\varepsilon=0$. 

A similar notation will be useful in the general case. Let $\rho=(\rho_1,\ldots,\rho_r)$ be a partition with weight $|\rho|\le n$. Let $\mathsf{M}^{(r)}$ denote the set of upper triangular $r \times r$ matrices with nonnegative integers, and $0$ on the 
diagonal. For any $1\le i<j \le r$ let $\varepsilon_{ij} \in \{0,2\}$, and define $\theta_{ij}=1$ if $\varepsilon_{ij}=0$ and $\theta_{ij}=\rho_i\rho_j$ otherwise. 

\begin{theorem} 
For $\mu=(\rho_1,\ldots,\rho_r,1^{n-|\rho|})$ we have
\begin{equation*}
(n)_{|\rho|}\, \hat{\chi}^\la_\mu=
\sum_{\varepsilon \in \{0,2\}^{r(r-1)/2}} \,
\sum_{(i_1,\ldots,i_r)\in \mathsf{N}^r}\,
 A_{i_1,\ldots,i_r}^{(\varepsilon)}(\rho_1,\ldots,\rho_r) \prod_{k=1}^r c^\la_{i_k}(\rho_k), 
\end{equation*}
with 
\begin{multline*}
A_{i_1,\ldots,i_r}^{(\varepsilon)}(\rho_1,\ldots,\rho_r)=
\sum_{a,b \in \mathsf{M}^{(r)}}\,\Bigg(\prod_{1\le i<j \le r} \theta_{ij} \binom{a_{ij}+1}{b_{ij}+1}\,
\frac{\rho_i(-\rho_i)^{b_{ij}}+{\rho_j}^{b_{ij}+1}}{\rho_i+\rho_j}\Bigg)\\\times
\prod_{k=1}^r s\Big(\rho_k+1,i_k+\sum_{l<k}(a_{lk}+\varepsilon_{lk})-\sum_{l>k} (a_{kl}-b_{kl})\Big),
\end{multline*}
and the convention that the sum on $a_{ij},b_{ij}$ is restricted to $a_{ij}=b_{ij}=0$ when $\varepsilon_{ij}=0$. 
\end{theorem} 
\begin{remark}
The right-hand side is a finite sum. Indeed for any $s\ge 1$, if we sum up the $r-s+1$ conditions
\[i_k+\sum_{l<k}(a_{lk}+\varepsilon_{lk})-\sum_{l>k} (a_{kl}-b_{kl})\le \rho_k+1\]
from $k=s$ to $k=r$, we obtain 
\[\sum_{k=s}^{r} \Big(i_k+\sum_{l<k}\varepsilon_{lk}+\sum_{l>k}b_{kl}+\sum_{l<s}a_{lk}\Big) \le \sum_{k=s}^{r} (\rho_k+1).\]
Hence any summation quantity remains bounded.
\end{remark}
\begin{proof}
Assuming the property true for $r-1$, we shall apply the Murnaghan rule under the form
\begin{equation*}
(n)_{|\rho|}\, \hat{\chi}^\la_\mu=\sum_{l=1}^{l(\la)} d_\la(l,\rho_1)\,
(n-\rho_1)_{|\rho|-\rho_1}\, \hat{\chi}^{\la -\rho_1\epsilon_l}_{\mu \setminus \rho_1}.
\end{equation*}
In general $\la -\rho_1\epsilon_l$ is not a partition. However by analytic continuation of a shifted symmetric polynomial, the recurrence assumption still writes
\begin{equation*}
(n-\rho_1)_{|\rho|-\rho_1}\,\hat{\chi}^{\la -\rho_1\epsilon_l}_{\mu \setminus \rho_1}=
\sum_{\varepsilon' \in \{0,2\}^{(r-1)(r-2)/2}} \,
\sum_{(i_2,\ldots,i_r)\in \mathsf{N}^{r-1}}\,
A_{i_2,i_3,\ldots,i_r}^{(\varepsilon')}(\rho_2,\ldots,\rho_r)
\prod_{k=2}^r \,c^{\la -\rho_1\epsilon_l}_{i_k}(\rho_k),
\end{equation*}
with $\varepsilon'$ standing for $\{\varepsilon_{ij}, 2\le i<j \le r\}$. By Lemma $1$ each $c^{\la -\rho_1\epsilon_l}_{i_k}(\rho_k)$ can be written
\begin{equation*}
\begin{split}
c^{\la -\rho_1\epsilon_l}_{i_k}(\rho_k)&=c^{\la}_{i_k}(\rho_k)+
\rho_1\rho_k\sum_{\begin{subarray}{c}r,s,t\ge 0 \\ r+s+t+2=i_k\end{subarray}}
c^{\la}_t(\rho_k) (\la_l-l+\rho_k+1)^r (\la_l-l-\rho_1+1)^s\\
&=c^{\la}_{i_k}(\rho_k)+
\rho_1\rho_k\sum_{\begin{subarray}{c}r,s,t,u,v\ge 0 \\ r+s+t+2=i_k\end{subarray}}c^{\la}_t(\rho_k)
\binom{r}{u}\binom{s}{v}(1+\rho_k)^{r-u} (1-\rho_1)^{s-v}\,(\la_l-l)^{u+v},
\end{split}
\end{equation*}
Defining
\[B_k(r,s,u,v)=(-1)^{u+v}\binom{r}{u}\binom{s}{v}(1+\rho_k)^{r-u} (1-\rho_1)^{s-v},\]
we obtain
\[c^{\la -\rho_1\epsilon_l}_{i_k}(\rho_k)=
\sum_{\varepsilon_{1k}\in \{0,2\}}\,
\theta_{1k}\sum_{\begin{subarray}{c}r_k,s_k,t_k\ge 0 \\ r_k+s_k+t_k+\varepsilon_{1k}=i_k\end{subarray}} c^{\la}_{t_k}(\rho_k)
\sum_{u_k,v_k\ge0}B_k(r_k,s_k,u_k,v_k)\,(-\la_l+l)^{u_k+v_k},\]
with $\theta_{1k}=1$ if $\varepsilon_{1k}=0$, $\theta_{1k}=\rho_1\rho_k$ if $\varepsilon_{1k}=2$, and the convention that the sum on $r_k,s_k$ is restricted to $r_k=s_k=0$ when $\varepsilon_{1k}=0$.
Inserting this expression in the recurrence assumption, we get
\begin{multline*}
(n-\rho_1)_{|\rho|-\rho_1}\,\hat{\chi}^{\la -\rho_1\epsilon_l}_{\mu \setminus \rho_1}=
\sum_{\varepsilon \in \{0,2\}^{r(r-1)/2}} \,
\sum_{(t_2,\ldots,t_r)\in \mathsf{N}^{r-1}}\,
\prod_{k=2}^r \, \theta_{1k}\, c^{\la}_{t_k}(\rho_k) \\\times
\sum_{\begin{subarray}{c}r_k,s_k\ge 0 \\ u_k,v_k\ge 0\end{subarray}}B_k(r_k,s_k,u_k,v_k)\,(-\la_l+l)^{u_k+v_k}
A_{z_2,z_3,\ldots,z_r}^{(\varepsilon')}(\rho_2,\ldots,\rho_r),
\end{multline*}
where for clarity of display, the notation $z_j$ $(2\le j \le r)$ stands for $z_j=r_j+s_j+t_j+\varepsilon_{1j}$. 
\newpage
Now we may insert this expression in the Murnaghan rule and apply Theorem 3. We obtain
\begin{multline*}
(n)_{|\rho|}\, \hat{\chi}^\la_\mu=
\sum_{\varepsilon \in \{0,2\}^{r(r-1)/2}} \,
\sum_{(t_2,\ldots,t_r)\in \mathsf{N}^{r-1}}\,
\prod_{k=2}^r \, \theta_{1k}\, c^{\la}_{t_k}(\rho_k) \\\times
\sum_{\begin{subarray}{c}r_k,s_k\ge 0 \\ u_k,v_k\ge 0\end{subarray}}B_k(r_k,s_k,u_k,v_k)\,
A_{z_2,z_3,\ldots,z_r}^{(\varepsilon')}(\rho_2,\ldots,\rho_r)
\\\times
\sum_{t_1,j\ge 0} (-1)^j \binom{\sum_{l\ge 2}(u_l+v_l)}{j} s(\rho_1+1,t_1-j)\,c^\la_{t_1}(\rho_1).
\end{multline*}
Writing $j=\sum_{l\ge 2}j_l$ and applying the Chu-Vandermonde formula, this can be rewritten
\begin{multline*}
(n)_{|\rho|}\, \hat{\chi}^\la_\mu=
\sum_{\varepsilon \in \{0,2\}^{r(r-1)/2}} \,
\sum_{(t_1,\ldots,t_r)\in \mathsf{N}^r}\,
c^\la_{t_1}(\rho_1) \prod_{k=2}^r c^\la_{t_k}(\rho_k)\, 
\theta_{1k}\\\times 
\sum_{\begin{subarray}{c}r_k,s_k,j_k\ge 0 \\ u_k,v_k\ge 0\end{subarray}}(-1)^{j_k} B_k(r_k,s_k,u_k,v_k)\binom{u_k+v_k}{j_k} \\\times
s\Big(\rho_1+1,t_1-\sum_{l>1}j_l\Big) \,
A_{z_2,z_3,\ldots,z_r}^{(\varepsilon')}(\rho_2,\ldots,\rho_r).
\end{multline*}
But by Lemma 2 we have
\[\sum_{\begin{subarray}{c}r_k,s_k,u_k,v_k\ge 0 \\ r_k+s_k=a_{1k}\end{subarray}}
B_k(r_k,s_k,u_k,v_k)\binom{u_k+v_k}{a_{1k}-b_{1k}} 
=(-1)^{a_{1k}-b_{1k}} \binom{a_{1k}+1}{b_{1k}+1}
\frac{\rho_k^{b_{1k}+1}-(-\rho_1)^{b_{1k}+1}}{\rho_k+\rho_1}.\] 
Writing $r_k+s_k=a_{1k}$ and $j_k=a_{1k}-b_{1k}$, $(2\le k \le r)$, we obtain
\begin{multline*}
(n)_{|\rho|}\, \hat{\chi}^\la_\mu=
\sum_{\varepsilon \in \{0,2\}^{r(r-1)/2}} \,
\sum_{(t_1,\ldots,t_r)\in \mathsf{N}^{r}}\,
\sum_{\begin{subarray}{c}(a_{12}\ldots a_{1r})\\ (b_{12}\ldots b_{1r})\end{subarray}}\,c^\la_{t_1}(\rho_1) \,\prod_{k=2}^r c^\la_{t_k}(\rho_k)\,
A_{z_2,z_3,\ldots,z_r}^{(\varepsilon')}(\rho_2,\ldots,\rho_r)
\\\times 
\theta_{1k}\, \binom{a_{1k}+1}{b_{1k}+1}
\frac{\rho_1(-\rho_1)^{b_{1k}}+\rho_k^{b_{1k}+1}}{\rho_1+\rho_k}
\,s\Big(\rho_1+1,t_1-\sum_{l>1}(a_{1l}-b_{1l})\Big),
\end{multline*}
where $z_j$ stands now for $z_j=t_j+a_{1j}+\varepsilon_{1j}$.
Hence the result.
\end{proof}

For $\rho_r=1$ either all $\varepsilon_{lr}, 1\le l \le r-1$, are equal to $0$, and the sum is restricted to $i_r=2$. Either only one $\varepsilon_{lr}$ is equal to $2$, and the sum is restricted to $i_r=0$. Or two or more $\varepsilon_{lr}$ are equal to $2$, such a contribution being zero. Summing up the $r$ non zero contributions brings a factor $(n-\sum_{i=1}^{r-1}\rho_i)$, and we recover the formula for $r-1$.

Ingram~\cite{In} had only computed the case $\rho=(2,2,2)$ (with many misprints). It writes
\begin{equation*}
\begin{split}
(n)_6\, \hat{\chi}^\la_{2,2,2,1^{n-6}}=&
\left(c_3^\la(2)-3 c_2^\la(2)\right)^3+
2\left(c_3^\la(2)-3c_2^\la(2)\right)\left(27c_3^\la(2)-47c_2^\la(2)-6c_4^\la(2)\right)\\
&+40c_5^\la(2)-240c_4^\la(2)+560c_3^\la(2)-600c_2^\la(2).
\end{split}
\end{equation*}
For $\rho=(3,2,2)$ a new example is
\begin{equation*}
\begin{split}
(n)_7\, \hat{\chi}&^\la_{3,2,2,1^{n-7}}=\\
&\left(\left(c_3^\la(2)-3c_2^\la(2)\right)^2-4c_4^\la(2)+18c_3^\la(2)-50c_2^\la(2)\right)\left(c_4^\la(3)-6c_3^\la(3)+11c_2^\la(3)\right)\\
&-12\left(c_3^\la(2)-3c_2^\la(2)\right)\left(c_5^\la(3)-8c_4^\la(3)+23c_3^\la(3)-28c_2^\la(3)\right)\\
&+72\left(c_6^\la(3)-10c_5^\la(3)+40c_4^\la(3)-80c_3^\la(3)+79c_2^\la(3)\right).
\end{split}
\end{equation*}

\begin{remark} 
As mentioned in Remark 3, since the order of the $\rho_k$ is irrelevant, there are many ways of writing  the right-hand side in terms of the $c_j^\la$ 's. Their equivalence yields identities between the $c_j^\la$ 's.
\end{remark}

\section{Jucys-Murphy elements}

We briefly recall the connection of our results with the structure of the center $\mathcal{Z}_n$ of the group algebra $\mathcal{C}S_n$ of $S_n$. These facts are not new, and may be found in~\cite[Section~4]{Co}.

Given a partition $\mu$, denote $\mathcal{C}_\mu$ the conjugacy class of permutations having cycle-type $\mu$, and identify this class with the formal sum of its elements. Then it is well known that we obtain a basis of $\mathcal{Z}_n$. 

For $1\le i\le n$ the Jucys-Murphy elements $J_i$ are defined by $J_i=\sum_{j<i} (ji)$, where $(ji)$ is a transposition. These elements were introduced independently in~\cite{Ju}  and~\cite{My}. They generate a maximal commutative subalgebra of $\mathcal{C}S_n$.

Jucys proved the two following fundamental properties. 
\begin{itemize}
\item[(i)]We have $\mathcal{S}[J_1,\ldots,J_n]=\mathcal{Z}_n$. More precisely the elementary symmetric functions $e_k(J_1,\ldots,J_n)$ are given by
\[e_k(J_1,\ldots,J_n)=\sum_{|\mu|-l(\mu)=k} \mathcal{C}_\mu.\]
\item[(ii)]Viewing any central function  $\chi$ as the formal sum $\sum_{\sigma}\chi(\sigma) \sigma$, for any symmetric function $f$, we have
\[f(J_1,\ldots,J_n)\, \chi^\la=f(A_\la)\, \chi^\la.\]
\end{itemize}

Now for any $\mu=(\rho,1^{n-|\rho|})$ with $\rho=(\rho_1,\ldots,\rho_r)$ having no part $1$, define  $\boldsymbol{z}_\mu=z_\rho$. As in Section 2.4, we write
\[(n)_{|\rho|}\, \hat{\chi}^\la_\mu=f_\mu(A_\la),\]
for a unique $f_\mu \in \mathsf{R}[\textrm{card},p_1,p_2,p_3,\ldots]$. Therefore the central character of $\la$ at $\mu$ is 
\[\omega^\la_\mu=n!\,{z}_\mu^{-1} \hat{\chi}^\la_\mu=\boldsymbol{z}_\mu^{-1}f_\mu(A_\la).\]
Thus we have
\[\mathcal{C}_\mu\, \chi^\la=\omega^\la_\mu\, \chi^\la = \boldsymbol{z}_\mu^{-1}f_\mu(A_\la)\, \chi^\la=\boldsymbol{z}_\mu^{-1}f_\mu(J_1,\ldots,J_n)\, \chi^\la,\]
which yields immediately
\[\mathcal{C}_\mu= \boldsymbol{z}_\mu^{-1}f_\mu(J_1,\ldots,J_n).\]

Since $f_\mu$ is unique, this expression of $\mathcal{C}_\mu$ as a symmetric function in the Jucys-Murphy elements is unique. 
Similarly from
\[e_k(A_\la)=\sum_{|\mu|-l(\mu)=k} \boldsymbol{z}_\mu^{-1}f_\mu(A_\la),\]
we obtain the decomposition
\[e_k=\sum_{|\mu|-l(\mu)=k} \boldsymbol{z}_\mu^{-1}f_\mu.\]
For instance we have
\[e_2= \frac{1}{3}f_{31^{n-3}}+\frac{1}{8}f_{221^{n-4}}=
p_2-\binom{n}{2}+\frac{1}{2}p_1^2-\frac{3}{2}p_2+\binom{n}{2}.\]

The expression of the power sums of the Jucys-Murphy elements $p_k(J_1,\ldots,J_n)$ in terms of the conjugacy classes $\mathcal{C}_\mu$ was studied in~\cite{LT}. To give a similar expression of $F_{npk}(J_1,\ldots,J_n)$ is an interesting open problem. 

\section{Application to Hecke algebras}

The Hecke algebra $H_n(q_1,q_2)$ is a deformation of the group algebra $\mathcal{C}S_n$. More precisely $H_n(q_1,q_2)$ is the algebra over $\mathsf{C}(q_1,q_2)$, the field of rational functions in two indeterminates $(q_1,q_2)$, generated by $T_1,T_2,...,T_{n-1}$ with relations 
\[(T_i-q_1)(T_i-q_2)=0,\quad \quad T_iT_{i+1}T_i =T_{i+1}T_iT_{i+1},
\quad \quad T_iT_j =T_jT_i, \quad \textrm{if}\quad |i-j|>1.\]

The case of $S_n$ corresponds to $q_1=-q_2=1$, and $T_i=s_i$, the simple transposition switching $i$ and $i+1$. The algebra $H_n(q_1,q_2)$ has a linear basis $\{T_\sigma, \sigma\in S_n\}$ defined by $T_\sigma=T_{s_{i_1}}T_{s_{i_2}}\ldots T_{s_{i_k}}$ where $s_{i_1}s_{i_2}\ldots s_{i_k}$ is any reduced decomposition of $\sigma$. 

Like those of $S_n$, the irreducible representations of $H_n(q_1,q_2)$ are indexed by partitions $|\la|=n$ and have dimension $\textrm{dim}\,\la$. We shall denote $\chi^\la_H$ (resp. $\hat{\chi}^\la_H$) the corresponding character (resp. normalized character).

For any permutation $\sigma\in S_n$ with cycle-type $\mu=(\mu_1,\mu_2,\ldots,\mu_l)$, we write $T_\mu=T_{\gamma_{\mu_1} \times \gamma_{\mu_2}\ldots \times \gamma_{\mu_l}}$, with $\gamma_k$ the $k$-cycle $s_{k-1}s_{k-2}\ldots s_1\in S_k$.
Then it is known~\cite[Corollary~5.2, p.~477]{R} that the characters $\chi^\la_H$ are uniquely determined by their values $\chi^\la_H(T_\mu), |\mu|=n$.

In the following we adopt ``$\la$-ring'' notations (see ~\cite[p.~471]{R} or~\cite[p.~220]{La4} for a short survey). In other words, given any symmetric function $f$ taken on an alphabet $A$, we shall write $f[(q_1+q_2)A]$ for the image of $f$ under the ring homomorphism uniquely determined by 
\[p_k[(q_1+q_2)A]=(q_1^k-(-q_2)^k)p_k(A)=(q_1^k-(-q_2)^k)\sum_{a\in A}a^k.\]

Then we have the generalized Frobenius formulas
\begin{equation*}
\begin{split}
(q_1+q_2)^{-l(\mu)} h_{\mu}[(q_1+q_2)A]&=\sum_{\la}\chi^\la_H(T_\mu) \, s_\la(A),\\
s_{\la}[(q_1+q_2)A]&=\sum_{\mu} (q_1+q_2)^{l(\mu)} \chi^\la_H(T_\mu) \, m_\mu(A),
\end{split}
\end{equation*}
with $m_\mu(A)$ the monomial symmetric function of $A$, i.e. the sum of distinct monomials $\prod_{i}a_i^{m_i}$ such that $(m_i)$ is a permutation of $\mu$. The first of these formulas was proved in~\cite[Theorem~4.14, p.~475]{R}, see also~\cite{RR}. For their equivalence, see~\cite{D}. 

By a classical result (see for instance~\cite[p.~238]{La4}), for any positive integer $r$ we have
\[h_{r}[(q_1+q_2)A]= \sum_{|\rho|=r}z_\rho^{-1} \prod_{i=1}^{l(\rho)} 
(q_1^{\rho_i}-(-q_2)^{\rho_i})\, p_\rho(A).\]
As a direct consequence, for $\mu=(\mu_1,\ldots,\mu_l)$ we have
\[\hat{\chi}^\la_H(T_\mu)= (q_1+q_2)^{-l} \sum_{|\rho^{(1)}|=\mu_1,\cdots,|\rho^{(l)}|=\mu_l}
\prod_{i=1}^{l} z_{\rho^{(i)}}^{-1}
\prod_{j=1}^{l(\rho^{(i)})}(q_1^{\rho^{(i)}_j}-(-q_2)^{\rho^{(i)}_j})\,
\hat{\chi}^\la_{\cup_i\rho^{(i)}},\]
with $\hat{\chi}^\la_{\cup_i\rho^{(i)}}$ given by Theorem $6$.
Here the partition $\mu\cup\nu$ is formed by parts of $\mu$ and $\nu$.

\section{Final remark}

In this paper the normalized characters of the symmetric group have been written in terms of the contents of the partition $\la$. A different approach has been recently devoted to the same problem, using ``Kerov polynomials''.

In this alternative framework the ``free cumulants'' of $\la$ are considered. These probabilistic quantities $R_i(\la),i \ge 2$ arise in the asymptotic study of the representations of symmetric groups~\cite{B1}. 

In the simplest case of a $p$-cycle $\mu=(p,1^{n-p})$, Kerov set the problem of writing the normalized character $(n)_p\, \hat{\chi}^\la_{p,1^{n-p}}$ as a polynomial in the free cumulants $R_i(\la)$. He conjectured that the coefficients of this polynomial are positive integers.

In \cite{B2,GR} this normalized character was obtained as a particular coefficient in some Taylor series. Special cases were also computed. However Kerov's positivity conjecture is still open. It would be interesting to study the connection between both descriptions (contents vs free cumulants).


\begin{thebibliography}{29}
\bibitem{B1}
P.\ Biane, \emph{Representations of symmetric groups and free probability}, Adv. Math. \textbf{138} (1998), 126Ð-181.
\bibitem{B2}
P.\ Biane, \emph{On the formula of Goulden and Rattan for Kerov polynomials}, S\'em. Lothar. Combin., \textbf{55} (2006), article B55d. 
\bibitem{Co}
S.\ Corteel, A.\ Goupil, G.\ Schaeffer, \emph{Content evaluation and class symmetric functions}, Adv. Math. \textbf{188} (2004), 315--336.
\bibitem{D}
J.\ D\'esarm\'enien, \emph{Une g\'en\'eralisation des caract\`eres du groupe sym\'etrique}, unpublished note (april 1996).
\bibitem{Fr}
G.\ Frobenius, \emph{\"Uber die Charaktere der Symmetrischen Gruppe}, S\"utzungsberichte der K\"oniglich Preussischen Akademie der Wissenschaften zu Berlin (1900), 516--534. Reprinted in \emph{Gessamelte Abhandlungen} \textbf{3}, 148Ð-166.
\bibitem{Ga}
A.\ Garsia, \emph{Young seminormal representation, Murphy elements and content evaluations}, lecture notes (march 2003), available at 
http://www.math.ucsd.edu/{\textasciitilde}garsia/ recentpapers/
\bibitem{Go}
D.\ M.\ Goldschmidt, \emph{Group characters, symmetric functions and the Hecke algebra}, University Lecture Series \textbf{4}, Amer. Math. Soc., Providence, 1991.
\bibitem{GR}
I.\ P.\ Goulden, A.\ Rattan, \emph{An explicit form for Kerov's character polynomials}, Trans. Amer. Math. Soc. to appear, ArXiv math.CO/0505317. 
\bibitem{In}
R.\ E.\ Ingram, \emph{Some characters of the symmetric group}, Proc. Amer. Math. Soc. \textbf{1} (1950), 358--369.
\bibitem{IO}
V.\ Ivanov, G.\ I.\ Olshanski, \emph{Kerov's central limit theorem for the Plancherel measure on Young diagrams}, in \emph{Symmetric functions 2001: Surveys of developments and perspectives}, 93--151, Kluwer, 2002. 
\bibitem{Ju}
A.\ A.\ Jucys, \emph{Symmetric polynomials and the center of the symmetric group ring}, Rep. Math. Phys. \textbf{5} (1974), 107--112.
\bibitem{Ka}
J.\ Katriel, \emph{Explicit expressions for the central characters of the symmetric group}, Discrete Applied Math. \textbf{67} (1996), 149-Ð156.
\bibitem{KO}
S.\ V.\ Kerov, G.\ I.\ Olshanski, \emph{Polynomial functions on the set of Young diagrams}, C.\ R.\ Acad.\ Sci.\ Paris S\'er.\ I, \textbf{319} (1994), 121--126.
\bibitem{Ls}
A.\ Lascoux, \emph{Notes on interpolation in one and several variables}, available at http://igm.univ-mlv.fr/{\textasciitilde}al/
\bibitem{LT}
A.\ Lascoux, J.-Y.\ Thibon, \emph{Vertex operators and the class algebras of symmetric groups}, Zapiski Nauchnyh Seminarov POMI, \textbf{283} (2001), 156--177.
\bibitem{La1}
M.\ Lassalle, \emph{Some combinatorial conjectures for Jack polynomials}, Ann. Comb. \textbf{2} (1998), 61--83.
\bibitem{La4}
M.\ Lassalle, \emph{Une $q$-sp\'ecialisation pour les fonctions
sym\'etriques monomiales}, Adv.\ Math., \textbf{162} (2001), 217--242.
\bibitem{La2}
M.Ö Lassalle, \emph{A new family of positive integers}, Ann. Comb. \textbf{6} (2002), 399--405.
\bibitem{La3}
M.\ Lassalle, \emph{Jack polynomials and some identities for partitions}, Trans. Amer. Math. Soc. \textbf{356} (2004), 3455--3476.
\bibitem{La5}
M.\ Lassalle, \emph{Explicitation of characters of the symmetric group}, C.\ R.\ Acad.\ Sci.\ Paris S\'er.\ I, \textbf{341} (2005), 529--534.
\bibitem{W}
M.\ Lassalle, available at http://igm.univ-mlv.fr/{\textasciitilde}lassalle/char.html
\bibitem{Ma}
I.\ G.\ Macdonald, \emph{Symmetric functions and Hall polynomials}, Clarendon Press, second edition, Oxford, 1995.
\bibitem{Mu}
F.\ D.\ Murnaghan, \emph{On the representations of the symmetric group}, Amer. J. Math. \textbf{59} (1937), 739Ð-753.
\bibitem{My}
G.\ E.\ Murphy, \emph{A new construction of Young's seminormal representation of the symmetric group}, J. Algebra \textbf{69} (1981), 287--291.
\bibitem{Na}
T.\ Nakayama, \emph{On some modular properties of irreducible representations of the symmetric group}, Jap. J. Math. \textbf{17} (1940), 165--184, 411--423.
\bibitem{OO}
A.\ Okounkov, G.\ I.\ Olshanski, \emph{Shifted Schur functions}, St. Petersburg Math. J. \textbf{9} (1998), 239--300.
\bibitem{R}
A.\ Ram, \emph{A Frobenius formula for the characters of the Hecke algebras}, Invent. Math. \textbf{106} (1991), 461-Ð488. 
\bibitem{RR}
A.\ Ram, J.\ B.\ Remmel, \emph{Applications of the Frobenius formulas for the characters of the symmetric group and the Hecke algebras of type A}, J. Alg. Comb. \textbf{6} (1997), 59-Ð87.
\bibitem{Su}
M.\ Suzuki, \emph{The values of irreducible characters of the symmetric group}, Amer. Math. Soc. Proceedings of Symposia in Pure Math. \textbf{47} (1987), 317--319.
\bibitem{VK}
A.\ M.\ Vershik, S.\ V.\ Kerov, \emph{Asymptotic theory of characters of symmetric groups}, Funct. Anal. Appl. \textbf{15} (1981), 246--255.
\bibitem{Z}
Jiang\ Zeng, private communication.

\end{thebibliography}
\end{document}